\documentclass{amsart}

\usepackage{amssymb}
\usepackage{latexsym}
\usepackage{amsmath}
\usepackage{euscript}
\usepackage{graphics}
\usepackage{tikz}

   \def\dZ{{\mathbb Z}}

   \def\cH{{\mathcal H}}

\def\bm\chi{\mbox{\boldmath$\chi$}}

\let\xker=\ker \def\ker{{\xker\,}}

\newtheorem{theorem}{Theorem}[section]

\theoremstyle{remark}

\numberwithin{equation}{section}

\author{Alexander I. Aptekarev}
\address{
Alexander I. Aptekarev\\
Keldysh Institute for Applied Mathematics\\
Russian Academy of Sciences\\
Miusskaya pl. 4\\
125047 Moscow, RUSSIA}

\author{Maxim Derevyagin}
\address{
Maxim Derevyagin\\
University of Mississippi\\
Department of Mathematics\\
Hume Hall 305 \\
P. O. Box 1848 \\
University, MS 38677-1848, USA }
\email{derevyagin.m@gmail.com}

\author{Walter Van Assche}
\address{
Walter Van Assche\\
KU Leuven\\
Department of Mathematics\\
Celestijnenlaan 200B box 2400\\
BE-3001 Leuven, BELGIUM}

\date{\today}


\begin{document}

\title[2D discrete Schr\"odinger operators]{On 2D discrete Schr\"odinger operators associated with multiple orthogonal polynomials}

\begin{abstract}
A class of cross-shaped difference operators on a two dimensional lattice is introduced.
The main feature of the operators in this class is that their formal eigenvectors consist of multiple orthogonal polynomials. In other words, this scheme generalizes the classical connection between Jacobi matrices and orthogonal polynomials to the case of operators on lattices. Furthermore we also show how to obtain 2D discrete Schr\"odinger operators out of this construction and give a number of explicit examples based on known families of multiple orthogonal polynomials.
\end{abstract}

\maketitle

\section{Introduction}

In this paper we introduce a class of cross-shaped difference operators acting on the lattice $\dZ^2_+$, where $\dZ_+=\{0,1,2,\dots\}$. Cross-shaped difference operators on lattices appear in many instances where some discrete systems are analysed. In particular, cross-shaped
difference operators with periodic coefficients were studied in \cite{K85} and \cite{OblPen}. To be more specific, the operators we are dealing with have the form
\begin{equation}\label{2DSchIntr}
(\widetilde{\Delta} f)_{n,m}=f_{n+1,m}+f_{n,m+1}+q_{n,m}f_{n,m}+a_{n,m}f_{n-1,m}+b_{n,m}f_{n,m-1}.
\end{equation}
This operator reflects the 2D interaction of the nearest neighbours on $\dZ^2$:
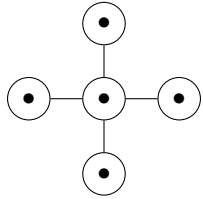
\begin{center}
\begin{tikzpicture}
  \tikzstyle{every node}=[draw,shape=circle];
  \node (v0) at (0:0) {$\bullet$};
  \node (v1) at (   0:1) {$\bullet$};
  \node (v2) at (  0:-1) {$\bullet$};
  \node (v3) at (90:1) {$\bullet$};
  \node (v4) at (3*90:1) {$\bullet$};
  \draw (v0) -- (v1)
        (v0) -- (v2)
        (v0) -- (v3)
        (v0) -- (v4);
\end{tikzpicture}
\end{center}
What we do can be considered as a generalization of
the classical connection between {\it  Jacobi matrices} and {\it
orthogonal polynomials} \cite{Ach1961}. Recall that a  Jacobi matrix $J$ is a
difference operator of the form
\begin{equation}\label{Jacobi}
(Jf)_{n}=\sqrt{a_{n+1}}f_{n+1}+q_{n}f_{n}+\sqrt{a_{n}}f_{n-1},
\end{equation}
which describes the 1D interaction of the nearest neighbours on $\dZ$:
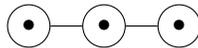
\begin{center}
\begin{tikzpicture}
  \tikzstyle{every node}=[draw,shape=circle];
  \node (v0) at (0:0) {$\bullet$};
  \node (v1) at (   0:1) {$\bullet$};
  \node (v2) at (  0:-1) {$\bullet$};
  \draw (v0) -- (v1)
        (v0) -- (v2);
\end{tikzpicture}
\end{center}
However, one should not think that the generalization in question is obvious and straightforward.  Unlike the classical case of Jacobi matrices,
it is not clear whether the corresponding eigenvalue problem
 \begin{equation}\label{2DSchEPIntr}
\widetilde{\Delta} \xi(z)=z\xi(z)
\end{equation}
has a solution and especially whether the entries of $\xi$ can be chosen to be
polynomials in the spectral variable $z$. To the best of our
knowledge, there were only a couple of operators on lattices with polynomial
eigenvectors known before and the main goal of this paper is to present a
rather general method to produce such operators. To this end, we construct
cross-shaped difference operators $\widetilde{\Delta}$ using multiple
orthogonal polynomials, which are known to play a prominent role in the theory
of random matrices  \cite{K2010}, \cite{AptKu}, \cite{Ken}. Thus, the operators
$\widetilde{\Delta}$ that we obtain in this way have multiple orthogonal
polynomials as the entries of their eigenvectors. In order to
guarantee the existence of a polynomial solution to \eqref{2DSchEPIntr} we only
consider special families of coefficients $q_{n,m}$, $a_{n,m}$, $b_{n,m}$ that
give rise to a {\it discrete zero curvature condition}. This means that there
is a {\it discrete integrable system} behind the scene \cite{ADvanA2014}.

It is evident that the difference expression \eqref{2DSchIntr}
defining $\widetilde{\Delta}$ is not symmetric. Nevertheless, in
some cases the operator $\widetilde{\Delta}$ can be symmetrized
like in the classical case of the Jacobi operator \eqref{Jacobi}. In
this case the cross-shaped difference operators are a subclass of a
wide class of operators known in the literature as discrete
electromagnetic Schr\"odinger operators (\cite{RabRoch},
\cite{Shubin}). {\it Discrete electromagnetic
Schr\"{o}\-dinger operators} are operators defined on the lattice
$\dZ^2$ that have the form
\begin{equation} \label{Electro}
\widetilde{\Delta}_s u = \sum_{k=1}^2 \frac{1}{2m_k} (V_{e_k}-a_k I) (V_{-e_k}
- \bar{a}_k I) u + \Phi u,
\end{equation}
where $u$ is a function defined on $\dZ^2$, $(V_{e_k}u)(x) := u(x-e_k)$ and $(V_{-e_k}u)(x) := u(x+e_k)$ are the shift operators by $e_1=(1,0)$ and
$e_2=(0,1)$, respectively, $m_k$ is the mass of the
$k$th particle, and $a_1$, $a_2$ and $\Phi$ are
bounded complex-valued functions on $\mathbb{Z}^2$. The vector-valued
function $a=(a_1, a_2)$ can be considered as an analogue
of the magnetic potential, whereas $\Phi$ is the discrete analogue of
the electric potential. If $\Phi$ and $a=(a_1, a_2)$ are real-valued, then $\widetilde{\Delta}_s$ is a
selfadjoint operator on the Hilbert space $\ell^2 (\mathbb{Z}^2)$ and it can be rewritten in the form
\begin{align*}
-\widetilde{\Delta}_s u(x)=&\frac{a_1}{2m_1}(u(x+e_1)+u(x-e_1))
+\frac{a_2}{2m_2}(u(x+e_2)+u(x-e_2))\\
&+\left(\frac{1+a_1^2}{2m_1}+\frac{1+a_2^2}{2m_2}+2\Phi(x)\right)u(x).
\end{align*}
Actually, in this paper we consider a class of operators that are more general than
the operator \eqref{Electro}. Namely, the operators we consider here are of the form
\begin{align*}
(\Delta_s f)_{n,m}=&\sqrt{\frac{a_{n+1,m}}{2}}f_{n+1,m}+\sqrt{\frac{b_{n,m+1}}{2}}f_{n,m+1}+\frac{c_{n,m}+d_{n,m}}{2}f_{n,m}\\
&+\sqrt{\frac{a_{n,m}}{2}}f_{n-1,m}+\sqrt{\frac{b_{n,m}}{2}}f_{n,m-1}.
\end{align*}
Here one has to put
\begin{align*}
f_{n,m}=u(x), \quad &f_{n+1,m}=u(x+e_1), \quad f_{n-1,m}=u(x-e_1),\\
&f_{n,m+1}=u(x+e_2), \quad f_{n,m-1}=u(x-e_2),
\end{align*}
to see the relation of our operators to the discrete electromagnetic
Schr\"{o}\-dinger operators \eqref{Electro}.

It is also worth mentioning that operators of the form \eqref{Electro}
describe the so-called tight binding model in solid state physics (see \cite{Mattis}, \cite{Mogilner} and the references given there), which plays a prominent role in the
theory of propagation of spin waves and of waves in quasi-crystals
\cite{Teschl,SR3}, in the theory of nonlinear integrable lattices
\cite{Teschl,Deift}, and in other places.
Furthermore such symmetric
difference operators appear in quantum-state transfer problems \cite{Miki} and quantum computation~\cite{Kit2010}.

On the one hand, our construction produces many concrete examples of cross-shaped difference operators, which can be obtained from some explicitly known families of
multiple orthogonal polynomials. For instance, these examples can serve for constructing certain Hamiltonians in quantum-state transfer problems and other related physical problems.  On the other hand, the eigenvectors of the operators in question consist of multiple orthogonal polynomials, whose asymptotic properties are well understood, and thus we have a very powerful tool for spectral analysis of the underlying operators.

\bigskip

\noindent{\bf Acknowledgements.} A.I. Aptekarev was supported by the Russain Science Foundation (project 14-21-00025. M. Derevyagin thanks the hospitality of the Department of Mathematics of KU Leuven, where his part of the research was mainly done while
he was a postdoc there. M. Derevyagin and W. Van Assche gratefully acknowledge the support of FWO Flanders project G.0934.13, KU Leuven research grant OT/12/073 and the Belgian Interuniversity Attraction Pole P07/18.

\section{Multiple orthogonal polynomials}

Here we briefly review a generalization of orthogonal polynomials to the case when we have two measures and we want our
polynomials to be simultaneously orthogonal with respect to the given measures, see \cite{Apt}, \cite[Chapter 23]{Ismail}, \cite{vanA1999}.

Given two positive measures $\mu_1$, $\mu_2$ on the real line, let us consider  the multi-index $(n,m) \in\dZ^2_+$. The
type II multiple orthogonal polynomial is the monic polynomial $P_{n,m}(x) = x^{n+m} + \cdots$
of degree $n+m$ such that the following orthogonality relations are satisfied:
\begin{eqnarray} \label{eq:1.9}
    \int P_{n,m}(x) x^j\, d\mu_1(x) &=& 0, \qquad j=0,1,\ldots,n-1, \nonumber \\
    \int P_{n,m}(x) x^j\, d\mu_2(x) &=& 0, \qquad j=0,1,\ldots,m-1. \nonumber
\end{eqnarray}
Introducing the moments
\[
s_j^{(i)} = \int x^j d\mu_i(x),\quad i=1,2,
\]
and the determinant of the moment matrix
\begin{equation} \label{eq:2.12}
   S_{n,m} = \left| \begin{matrix}
                   s_0^{(1)} & s_1^{(1)} & \cdots & s_{n-1}^{(1)} \\
                   s_1^{(1)} & s_2^{(1)} & \cdots & s_{n}^{(1)} \\
                \vdots & \vdots & \cdots & \vdots \\
                   s_{n+m-1}^{(1)} & s_{n+m}^{(1)} & \cdots & s_{2n+m-2}^{(1)} \end{matrix} \
           \begin{matrix}
                   s_0^{(2)} & s_1^{(2)} & \cdots & s_{m-1}^{(2)} \\
                   s_1^{(2)} & s_2^{(2)} & \cdots & s_{m}^{(2)} \\
                \vdots & \vdots & \cdots & \vdots \\
                   s_{n+m-1}^{(2)} & s_{n+m}^{(2)} & \cdots & s_{n+2m-2}^{(2)}
                  \end{matrix} \right|,
\end{equation}
one sees that the type II multiple orthogonal polynomial can be written as
\small
\[    P_{n,m}(x) = \frac{1}{S_{n,m}}
          \left|\begin{matrix}
                   s_0^{(1)} & s_1^{(1)} & \cdots & s_{n-1}^{(1)} \\
                   s_1^{(1)} & s_2^{(1)} & \cdots & s_{n}^{(1)} \\
                \vdots & \vdots & \cdots & \vdots \\
                   s_{n+m}^{(1)} & s_{n+m+1}^{(1)} & \cdots & s_{2n+m-1}^{(1)} \end{matrix} \
           \begin{matrix}
                   s_0^{(2)} & s_1^{(2)} & \cdots & s_{m-1}^{(2)} \\
                   s_1^{(2)} & s_2^{(2)} & \cdots & s_{m}^{(2)} \\
                \vdots & \vdots & \cdots & \vdots \\
                   s_{n+m}^{(2)} & s_{n+m+1}^{(2)} & \cdots & s_{n+2m-1}^{(2)}
                  \end{matrix}
          \begin{matrix} 1 \\ x \\ \vdots \\ x^{n+m} \end{matrix} \right|
\]
\normalsize
provided that $S_{n,m}$ is nonvanishing. In the latter case we say that the index $(n,m)$ is normal.
In this paper we always assume that
all multi-indices are normal and we investigate the nearest-neighbor recurrence relations.

\begin{theorem}[\cite{vanA2011}] \label{thm:1}
Suppose all multi-indices $(n,m) \in\dZ_+^2$ are normal.
Then the type II multiple orthogonal polynomials satisfy the system of recurrence
relations
\begin{align}
    P_{n+1,m}(x) &= (x-c_{n,m})P_{n,m}(x) - a_{n,m} P_{n-1,m}(x) - b_{n,m} P_{n,m-1}(x), \label{eq:2.6} \\
    P_{n,m+1}(x) &= (x-d_{n,m})P_{n,m}(x) - a_{n,m} P_{n-1,m}(x) - b_{n,m} P_{n,m-1}(x), \label{eq:2.7}
\end{align}
where the coefficients obey the following conditions
\begin{equation}\label{BVC_dis}
a_{0,m}=b_{n,0}=0,\, a_{n,0}>0, \, b_{0,m}>0, \quad n,m > 0.
\end{equation}
\end{theorem}

Note that for \eqref{eq:2.6}, \eqref{eq:2.7} to hold these relations must
be consistent. As was shown in \cite{ADvanA2014}, this consistency can be expressed as
a discrete zero curvature condition, which also takes the following form.

\begin{theorem}[\cite{vanA2011, ADvanA2014}] \label{thm1}
Suppose that all the indices $(n,m) \in \mathbb{Z}^2_+$ are normal.
The recurrence coefficients in the recurrence relations \eqref{eq:2.6}--\eqref{eq:2.7} for
type II multiple orthogonal polynomials satisfy the following equations
\begin{align}
    c_{n,m+1}=&c_{n,m}\,+\,\displaystyle\frac{(a+b)_{n+1,m}\,-\,(a+b)_{n,m+1}}{(c-d)_{n,m}},   \label{eq:3.1}  \\
    d_{n,m+1}=&d_{n,m}\,+\,\displaystyle\frac{(a+b)_{n+1,m}\,-\,(a+b)_{n,m+1}}{(c-d)_{n,m}},
  \qquad\qquad \label{eq:3.2} \\
    \frac{a_{n,m+1}}{a_{n,m}} =& \frac{c_{n,m}-d_{n,m}}{c_{n-1,m}-d_{n-1,m}},  \label{eq:3.3} \\
    \frac{b_{n+1,m}}{b_{n,m}} =& \frac{c_{n,m}-d_{n,m}}{c_{n,m-1}-d_{n,m-1}}.  \label{eq:3.4}
\end{align}
\end{theorem}

It turns out that the consistency conditions \eqref{eq:3.1}-\eqref{eq:3.4}, which generate
an underlying zero curvature condition, play a central role in the theory of multiple orthogonal polynomials in the sense that
the following Favard-type result holds.
\begin{theorem}\label{FavardMP}
Suppose that the polynomials $P_{n,m}$ of degree $n+m$ satisfy \eqref{eq:2.6}, \eqref{eq:2.7} and
for the corresponding coefficients the consistency conditions \eqref{eq:3.1}-\eqref{eq:3.4} and \eqref{BVC_dis} are fulfilled.
Then there are two measures $\mu_1$ and $\mu_2$ such that the polynomials $P_{n,m}$ are multiple
orthogonal polynomials with respect to $\mu_1$ and $\mu_2$.
\end{theorem}

\section{A continuous model for multiple orthogonal polynomials}\label{ToySection}

In this section we propose a way to interpret the concept of multiple orthogonal polynomials.
Let us begin by recalling that the discretization of the string equation
\[
\frac{d}{dM(x)}\frac{dy(x)}{dx}+q(x)y(x)=zy(x),
\]
where $\frac{d}{dM(x)}$ is the derivative with respect to the measure $dM(x)$, leads to spectral problems for Jacobi matrices and thus to orthogonal polynomials (for instance see \cite{Atk1964}, \cite{Strikwerda}).
In particular, if we try to solve the differential equation
\[
y''(x)=zy(x)
\]
by using the classical discretization scheme
\[
y''(x)\approx y(x+1)-2y(x)+y(x-1),
\]
we get the difference equation
\[
y(x+1)+y(x-1)=(z+2)y(x),
\]
which brings us to the context of the free Jacobi matrix and Chebyshev polynomials.

Bearing the above-mentioned trick in mind,
one can see that the discretization of the following space-time dual system of generalized string equations
\begin{align*}
-\Psi_t(t,x)+\frac{\partial}{\partial M_1(x)}\frac{\partial\Psi(t,x)}{\partial x}+u(t,x)\Psi(t,x)=z\Psi(t,x),\\
-\Psi_x(t,x)+\frac{\partial}{\partial M_2(t)}\frac{\partial\Psi(t,x)}{\partial t}+v(t,x)\Psi(t,x)=z\Psi(t,x),
\end{align*}
gives a system of difference equations of the form \eqref{eq:2.6}--\eqref{eq:2.7}. Note that such generalized string equations were recently studied in \cite{AlMB}.

To get a more precise idea, let us consider the following space-time dual system of time-dependent Schr\"odinger equations
\begin{align*}
-\Psi_t(t,x)+\Psi_{xx}(t,x)+u(t,x)\Psi(t,x)=z\Psi(t,x), \\
-\Psi_x(t,x)+\Psi_{tt}(t,x)+v(t,x)\Psi(t,x)=z\Psi(t,x),
\end{align*}
where $t$, $x$ are nonnegative real numbers, $z$ is the spectral parameter, and
$u$, $v$ are sufficiently good potentials.
Now let us fix $h>0$. Then using the following discretization for the derivatives of the first order
\[
\Psi_t(t,x)\approx\frac{\Psi(t-h,x)-\Psi(t,x)}{-h},\quad \Psi_x(t,x)\approx\frac{\Psi(t,x-h)-\Psi(t,x)}{-h},
\]
and for the derivatives of the second order
\begin{align*}
\Psi_{xx}(t,x)\approx\frac{\Psi(t,x+h)-2\Psi(t,x)+\Psi(t,x-h)}{h^2},\\
\Psi_{tt}(t,x)\approx\frac{\Psi(t+h,x)-2\Psi(t,x)+\Psi(t-h,x)}{h^2},
\end{align*}
we arrive at the following recurrence relations
\begin{align*}
\frac{\Psi(t-h,x)-\Psi(t,x)}{h}+\frac{\Psi(t,x+h)-2\Psi(t,x)+\Psi(t,x-h)}{h^2}+u(t,x)\Psi(t,x)=z\Psi(t,x),\\
\frac{\Psi(t,x-h)-\Psi(t,x)}{h}+\frac{\Psi(t+h,x)-2\Psi(t,x)+\Psi(t-h,x)}{h^2}+v(t,x)\Psi(t,x)=z\Psi(t,x).
\end{align*}
Choosing $x=t=0$, $h=1$ and setting
\[
P_{n,m}(z)=\Psi(nt,mh), \quad u_{n,m}=u(nt,mh)-3, \quad v_{n,m}=v(nt,mh)-3,\quad
m,n\in\dZ_+
\]
we get from the latter recurrence relations that
\begin{eqnarray*}
P_{n,m+1}(z) +u_{n,m}P_{n,m}(z) + P_{n-1,m}(z) + P_{n,m-1}(z)=zP_{n,m}(z), \\
    P_{n+1,m}(z)+ v_{n,m}P_{n,m}(z) + P_{n-1,m}(z) + P_{n,m-1}(z)=zP_{n,m}(z),
\end{eqnarray*}
which are the nearest neighbour recurrence relations for multiple orthogonal polynomials. However, as we already learned, the coefficients of the nearest neighbour recurrence relations cannot be arbitrary. As a consequence, the fact that a few coefficients
in the just obtained relations are equal to 1 makes our choice for the rest trivial. Indeed, it is easy
to see that only for the case of constant coefficients $u_{n,m}=u_{0,0}$ and $v_{n,m}=v_{0,0}\ne u_0$ the consistency conditions  
\eqref{eq:3.1}--\eqref{eq:3.4} are satisfied.

\section{The underlying pairs of operators}

In analogy with the case of orthogonal polynomials on the real
line, we introduce two difference operators on $\dZ_+^2$
associated with the recurrence relations \eqref{eq:2.6}--\eqref{eq:2.7},
whose coefficients  obey the discrete integrable
system \eqref{eq:3.1}--\eqref{eq:3.4},
\begin{equation}\label{DifOp1}
(H_1f)_{n,m}=f_{n+1,m}+c_{n,m}f_{n,m}+a_{n,m}f_{n-1,m}+b_{n,m}f_{n,m-1},
\end{equation}
\begin{equation}\label{DifOp2}
(H_2f)_{n,m}=f_{n,m+1}+d_{n,m}f_{n,m}+a_{n,m}f_{n-1,m}+b_{n,m}f_{n,m-1},
\end{equation}
where $f=\bigl(f_{n,m}\bigr)$ is a sequence defined on $\dZ_+^2$. Then it is clear that \eqref{eq:2.6}--\eqref{eq:2.7}
can be rewritten as the eigenvector problems
\begin{equation}\label{SepCond}
\begin{split}
H_1\pi(z)=z\pi(z),\\
H_2\pi(z)=z\pi(z),
\end{split}
\end{equation}
where $\pi(z)=\bigl(P_{n,m}(z)\bigr)$ is a table of multiple orthogonal polynomials.
In order to imagine what these relations represent, let us rewrite the operators
$H_1$ and $H_2$ in the following way:
\begin{align*}
(H_1f)_{n,m}&=f_{n+1,m}+(c_{n,m}+b_{n,m})f_{n,m}+a_{n,m}f_{n-1,m}+b_{n,m}
(f_{n,m-1}-f_{n,m}), \\
(H_2f)_{n,m}&=f_{n,m+1}+(d_{n,m}+a_{n,m})f_{n,m}+b_{n,m}f_{n,m-1}+
a_{n,m}(f_{n-1,m}-f_{n,m}).
\end{align*}
Now, recalling the interpretation from Section \ref{ToySection} with time-dependent
Schr\"odinger equations, one can think of \eqref{SepCond} as the relations that
describe two coexisting evolutions: one is the transformation of the vector
$(P_{0,m},P_{1,m},\dots)$ in the discrete time $m$ and the other one is the progression of the vector $(P_{n,0},P_{n,1},\dots)$ in the discrete time $n$.
In other words, one could visualize this as two waves going from the boundaries
$\dZ_+$ and $i\dZ_+=(0,m)$, $m\in\dZ_+$, to infinity along $i\dZ_+$ and $\dZ_+$, respectively. These ideas suggest that the boundary data play a crucial role for the theory.
For this reason we introduce two monic classical Jacobi matrices
\begin{equation*}
\begin{split}
(\cH_1f)_{n,0}&=f_{n+1,0}+c_{n,0}f_{n,0}+a_{n,0}f_{n-1,0},
\\
(\cH_2f)_{0,m}&=f_{0,m+1}+d_{0,m}f_{0,m}+b_{0,m}f_{0,m-1},
\end{split} \qquad a_{n,0},\,b_{n,0}>0\,,
\end{equation*}
which, due to the fact that $a_{0,m}=a_{n,0}=0$ for $n,m\in\dZ_+$, are restrictions of $H_1$ and $H_2$ to the subspaces spanned by functions defined on
$\dZ_+$ and $i\dZ_+$, respectively. Observe that the entire initial information about multiple orthogonal polynomials is encrypted in the matrices $\cH_1$ and $\cH_2$.

\begin{theorem}\label{Unique}
The Jacobi matrices $\cH_1$ and $\cH_2$ (and as a consequence the operators $H_1$ and $H_2$) determine the measures $\mu_1$ and $\mu_2$, respectively.
In other words, the solution of the discrete integrable system \eqref{eq:3.1}--\eqref{eq:3.4}  can be reconstructed from the boundary data.
\end{theorem}
\begin{proof}
The proof of this statement is straightforward and it is enough to notice that the polynomials
$P_{n,0}$ and $P_{0,m}$ are orthogonal polynomials associated with $\cH_1$ and $\cH_2$. Thus, it remains to apply the classical Favard theorem (see \cite[Section 2.5]{Ismail}) to determine the measures $\mu_1$ and $\mu_2$.
\end{proof}

\noindent \textbf{Remark.}
To sum up what we have so far we note that Theorem \ref{Unique} says that starting with
the nearest neighbour recurrence relations one can reconstruct the underlying measures and, consequently, the corresponding sequences of moments. If we go in the opposite direction then we start with the moments. Next, we find the multiple orthogonal polynomials and after that we end up with the recurrence coefficients. Hence, we know how to solve inverse and direct problems
for $H_1$ and $H_2$. However, one should not think that any two measures can be put into this scheme. As a matter of fact, any two measures define two Jacobi matrices, i.e., two 1D discrete Schr\"odinger operators on $\dZ_+$. In the standard way, these two operators can be glued into one operator  on
$\dZ$ (e.g, one can take a $2\times 2$ block diagonal operator with these two Jacobi matrices on the diagonal), which can be equivalently interpreted as a 1D discrete Schr\"odinger operator on $\dZ_+\cup i\dZ_+$. Still, it remains unclear whether this operator on $\dZ_+\cup i\dZ_+$ can be extended to $\dZ_+^2$ in one way or another (or, perhaps, our two initial 1D operators can be glued into one 2D operator).  One can see that the problem of the existence of a table of multiple orthogonal polynomials is thus equivalent to the possibility of extending the 1D discrete Schr\"odinger operator on 
$\dZ_+\cup i\dZ_+$ to a 2D discrete Schr\"odinger operator on $\dZ_+^2$. It is worth mentioning that this extension can be done
if and only if the system $(\mu_1, \mu_2)$ is a {\it perfect system} (see \cite{ADvanA2014} for further details).

\section{Cross-shaped difference operators on $\dZ_+^2$}

It is obvious that for a general cross-shaped difference operator
$\widetilde{\Delta}$
of the form
\begin{equation}\label{2DSch}
(\widetilde{\Delta} f)_{n,m}=\frac12 f_{n+1,m}+\frac12 f_{n,m+1}+q_{n,m}f_{n,m}+a_{n,m}f_{n-1,m}+b_{n,m}f_{n,m-1}
\end{equation}
it is not at all clear whether the eigenvalue problem
\begin{equation}\label{EigenProblem}
\widetilde{\Delta}\xi(z)=z\xi(z)
\end{equation}
has a polynomial solution.

If one supposes that \eqref{EigenProblem} has a polynomial solution then one gets that the relation
\[
\frac12 f_{1,0}+\frac12 f_{0,1}+q_{0,0}f_{0,0}+a_{0,0}f_{-1,0}+b_{0,0}f_{0,-1}=
zf_{0,0},
\]
with the initial conditions
\[
f_{-1,0}=0, \quad f_{0,-1}=0, \quad f_{0,0}=1,
\]
must define two linear monic polynomials $f_{0,1}$ and $f_{1,0}$. This basically means that there exists a representation
\[
q_{0,0}=\frac{c_{0,0}}{2}+\frac{d_{0,0}}{2}
\]
such that
\[
f_{0,1}=z-{d_{0,0}}, \quad
f_{1,0}=z-{c_{0,0}}.
\]
In general, we see that the relation
\[
\frac12 f_{n+1,m}+\frac12 f_{n,m+1}+q_{n,m}f_{n,m}+a_{n,m}f_{n-1,m}+b_{n,m}f_{n,m-1}=zf_{n,m}
\]
determines two polynomials $f_{n+1,m}$ and $f_{n+1,m}$. In other words, there are two representations
\begin{align*}
f_{n+1,m}+q'_{n,m}f_{n,m}+a'_{n,m}f_{n-1,m}+b'_{n,m}f_{n,m-1}&=zf_{n,m},\\
f_{n,m+1}+q''_{n,m}f_{n,m}+a''_{n,m}f_{n-1,m}+b''_{n,m}f_{n,m-1}&=zf_{n,m},
\end{align*}
that must be consistent on $\dZ_+^2$. Schematically, what we do here is to try to split the
2D interaction in question as the arithmetic mean of the following two interactions:
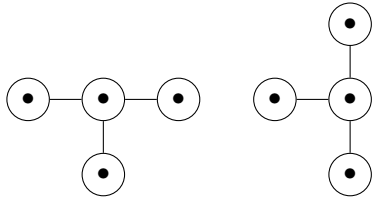
\begin{center}
\begin{tikzpicture}
  \tikzstyle{every node}=[draw,shape=circle];
  \node (v0) at (0:0) {$\bullet$};
  \node (v1) at (   0:1) {$\bullet$};
  \node (v2) at (  0:-1) {$\bullet$};
  \node (v4) at (3*90:1) {$\bullet$};
  \draw (v0) -- (v1)
        (v0) -- (v2)
        (v0) -- (v4);
\end{tikzpicture}\quad\quad
\begin{tikzpicture}
  \tikzstyle{every node}=[draw,shape=circle];
  \node (v0) at (0:0) {$\bullet$};
  \node (v2) at (  0:-1) {$\bullet$};
  \node (v3) at (90:1) {$\bullet$};
  \node (v4) at (3*90:1) {$\bullet$};
  \draw 
        (v0) -- (v2)
        (v0) -- (v3)
        (v0) -- (v4);
\end{tikzpicture}
\end{center}
Since it is a very difficult problem to characterize all such representations, we use the following multiple polynomial Ansatz to proceed:
\[
q'_{n,m}=c_{n,m}, \quad q''_{n,m}=d_{n,m}, \quad
a'_{n,m}=a''_{n,m}=a_{n,m}, \quad b'_{n,m}=b''_{n,m}=b_{n,m}.
\]

In order to be able to present a class of operators for which the eigenvalue problem has a polynomial solution, we suppose that
\begin{equation}\label{TechCondDet}
D_{n,m}=-4\left(\frac{a_{n+1,m+1}}{a_{n+1,m}}+\frac{b_{n+1,m+1}}{b_{n,m+1}}\right)+8\ne 0.
\end{equation}
Then we introduce the matrices
\begin{equation} \label{L}
L_{n,m}=\begin{pmatrix}  z-q_{n,m}+4\frac{q_{n+1,m}}{D_{n,m}}-4\frac{q_{n+1,m}}{D_{n,m}} & -a_{n,m} & -b_{n,m} \\
                                 1 & 0 & 0 \\
                                 1 & 0 &  8\left(\frac{q_{n+1,m-1}}{D_{n,m-1}}-
                                 \frac{q_{n+1,m-1}}{D_{n,m-1}}\right)  \end{pmatrix},
\end{equation}
\begin{equation}  \label{M}
M_{n,m}=\begin{pmatrix}  z-q_{n,m}-4\frac{q_{n+1,m}}{D_{n,m}}+4\frac{q_{n+1,m}}{D_{n,m}} & -a_{n,m} & -b_{n,m} \\
                                 1 &  -8\left(\frac{q_{n,m}}{D_{n-1,m}}-
                                 \frac{q_{n,m}}{D_{n-1,m}}\right) & 0 \\
                                 1 & 0 & 0  \end{pmatrix}.
\end{equation}
We associate these transition matrices $L_{n,m}$ and $M_{n,m}$ with the operator $\widetilde{\Delta}$. These matrices allow us to define a vector wave function $\Psi_{n,m}(z)$ on $\dZ_+^2$:
\[
\Psi_{n+1,m}(z)=L_{n,m}(z)\Psi_{n,m}(z), \quad \Psi_{n,m+1}(z)=M_{n,m}(z)\Psi_{n,m}(z).
\]
If this function is correctly defined, then by choosing the initial state
\[
   \Psi_{0,0}(z) = \begin{pmatrix} 1\\ 0\\ 0 \end{pmatrix}
\]
we arrive at the polynomial solution
\[   \Psi_{n,m} =   \begin{pmatrix} P_{n,m}(z) \\ P_{n-1,m}(z) \\ P_{n,m-1}(z) \end{pmatrix},  \]
which consists of multiple orthogonal polynomials and at the same time gives a polynomial solution to the eigenvalue problem for $\widetilde{\Delta}$.

\begin{theorem}
Let $\widetilde{\Delta}$ be a  cross-shaped difference operator of
the form \eqref{2DSch} such that the condition \eqref{TechCondDet}
is satisfied. Then the eigenvalue problem for $\widetilde{\Delta}$
has a family of multiple orthogonal polynomials as its solution if
and only if the following discrete zero curvature condition holds
\begin{equation}\label{Lax}
L_{n,m+1}M_{n,m}-M_{n+1,m}L_{n,m}=0, \quad n,m\in\dZ,
\end{equation}
where $L_{n,m}$ and $M_{n,m}$ are the matrices in \eqref{L}--\eqref{M}.
\end{theorem}
\begin{proof}
If we have a family of multiple orthogonal polynomials then
we can introduce
the following cross-shaped difference operator
\[
\Delta=\frac{1}{2}(H_1+H_2).
\]
It is easy to see that the action of $\Delta$ on $f$ can be described by the following 
\[
(\Delta f)_{n,m}=\frac{1}{2}f_{n+1,m}+\frac{1}{2}f_{n,m+1}+\frac{d_{n,m}+c_{n,m}}{2}f_{n,m}+a_{n,m}f_{n-1,m}%
+b_{n,m}f_{n,m-1}.
\]
Thus, from \eqref{SepCond} we get that
\[
\Delta\pi(z)=z\pi(z).
\]
To prove this statement we first have to show how to reconstruct the coefficients $c_{n,m}$, $d_{n,m}$ from
the $a_{n,m}$, $b_{n,m}$, and
\[
q_{n,m}=\frac{c_{n,m}+d_{n,m}}{2}.
\]
To begin, we observe that the above relation
gives the following three equations
\begin{equation}\label{LineEqPart1}
\begin{split}
{c_{n,m}+d_{n,m}}&=2q_{n,m},\\
{c_{n+1,m}+d_{n+1,m}}&=2q_{n+1,m},\\
{c_{n,m+1}+d_{n,m+1}}&=2q_{n,m+1},
\end{split}
\end{equation}
where $c_{n,m}$, $d_{n,m}$, $c_{n+1,m}$, $d_{n+1,m}$, $c_{n,m+1}$, $d_{n,m+1}$ are six unknowns. Next, the relations
\eqref{eq:3.1}, \eqref{eq:3.3}, and \eqref{eq:3.4} give
\begin{equation}\label{LineEqPart2}
\begin{split}
c_{n,m}-d_{n,m}+d_{n+1,m}-c_{n,m+1}&=0,\\
\frac{a_{n+1,m+1}}{a_{n+1,m}}c_{n,m}-\frac{a_{n+1,m+1}}{a_{n+1,m}}d_{n,m}-c_{n+1,m}+d_{n+1,m}&=0,\\
\frac{b_{n+1,m+1}}{b_{n,m+1}}c_{n,m}-\frac{b_{n+1,m+1}}{b_{n,m+1}}d_{n,m}-c_{n,m+1}+d_{n,m+1}&=0.
\end{split}
\end{equation}
Therefore, we arrive at the system \eqref{LineEqPart1}--\eqref{LineEqPart2} of six linear equations with six unknowns.
The determinant of \eqref{LineEqPart1}--\eqref{LineEqPart2}  is
\[
\begin{vmatrix}
1&1&0&0&0&0\\
0&0&1&1&0&0\\
0&0&0&0&1&1\\
1&-1&0&1&-1&0\\
\frac{a_{n+1,m+1}}{a_{n+1,m}}&-\frac{a_{n+1,m+1}}{a_{n+1,m}}&-1&1&0&0\\
\frac{b_{n+1,m+1}}{b_{n,m+1}}&-\frac{b_{n+1,m+1}}{b_{n,m+1}}&0&0&-1&1\\
\end{vmatrix}=
-4\left(\frac{a_{n+1,m+1}}{a_{n+1,m}}+\frac{b_{n+1,m+1}}{b_{n,m+1}}\right)+8.
\]
Hence, due to \eqref{TechCondDet} we see that one can solve \eqref{LineEqPart1}, \eqref{LineEqPart2} and find
$c_{n,m}$, $d_{n,m}$ provided that $a_{n,m}$, $b_{n,m}$, $q_{n,m}$ are given:
\[
c_{n,m}=q_{n,m}-4\frac{q_{n+1,m}}{D_{n,m}}+4\frac{q_{n+1,m}}{D_{n,m}},\quad
d_{n,m}=q_{n,m}+4\frac{q_{n+1,m}}{D_{n,m}}-4\frac{q_{n+1,m}}{D_{n,m}}.
\]
The latter relations show that the relation \eqref{Lax} is equivalent to the consistency conditions \eqref{eq:3.1}--\eqref{eq:3.4}, which are in turn the necessary and sufficient conditions for a family of multiple orthogonal
polynomials to exist according to Theorem \ref{FavardMP}.
\end{proof}

\noindent\textbf{Remark.}
As one can see from the example constructed in Section \ref{ToySection}, if the condition \eqref{TechCondDet} is not satisfied
then the operators $H_1$, $H_2$ can not be uniquely determined. An algorithm for computing the coefficients $a_{n,m},b_{n,m},c_{n,m},d_{n,m}$
from the coefficients of the operators $H_1$ and $H_2$ is given in \cite{FiHaVA}.

\section{2D discrete Schr\"odinger operators}

In a natural way one can introduce the space $\ell^2(\dZ_+^2)$ of square summable families on $\dZ^2_+$. Moreover,
as in the case of Jacobi matrices, the difference expression $\Delta$ generates an operator in $\ell^2(\dZ_+^2)$, which will also be denoted
by $\Delta$. It is clear that $\Delta$ is not symmetric. However, mimicking the idea of the transformation of monic Jacobi matrices
to symmetric ones, we can symmetrize $\Delta$ by making use of the consistency conditions \eqref{eq:3.1}--\eqref{eq:3.4}.

\begin{theorem}
Suppose that in addition to the curvature condition \eqref{eq:3.1}--\eqref{eq:3.4}, the coefficients
$c_{n,m}$ and $d_{n,m}$ verify the condition
\begin{equation}\label{ForSymToEx}
{c_{n+1,m}-d_{n+1,m}}=
{c_{n,m+1}-d_{n,m+1}}.
\end{equation}
Then there exists a family $h_{n,m}\ne 0$ defined on $\dZ^2_+$ such that the diagonal operator
\[
(Df)_{n,m}=h_{n,m}f_{n,m}
\]
symmetrizes the operator $\Delta$ by means of a similarity transformation, that is,  the operator
\[
\Delta_s=D^{-1}\Delta D
\]
is symmetric in $\ell^2(\dZ_+^2)$.
\end{theorem}
\begin{proof}
Note that if a difference expression of the form
\[
(\widehat{\Delta} f)_{n,m}=\alpha_{n,m}f_{n+1,m}+\beta_{n,m}f_{n,m+1}+\gamma_{n,m}f_{n,m}+
\delta_{n,m}f_{n-1,m}+\varepsilon_{n,m}f_{n,m-1}
\]
is symmetric in $\ell^2(\dZ_+^2)$ then
\begin{equation}\label{GenSym}
\alpha_{n,m}=\delta_{n+1,m},\quad \beta_{n,m}=\varepsilon_{n,m+1},
\end{equation}
which can be obtained from the relation
\[
\left(\widehat{\Delta} f,g\right)_{\ell^2(\dZ_+^2)}=\left(f,\widehat{\Delta}g\right)_{\ell^2(\dZ_+^2)}
\]
on the standard basis $(e_{n,m})$, where $(e_{n,m})$ is understood as a table of numbers with 1 on the position $(n,m)$ and the rest of the elements are zeros. Since we want the operator
\begin{align*}
(\Delta_s f)_{n,m}=&\frac{h_{n+1,m}}{2h_{n,m}}f_{n+1,m}+\frac{h_{n,m+1}}{2h_{n,m}}f_{n,m+1}+\frac{c_{n,m}+d_{n,m}}{2}f_{n,m}\\
&+a_{n,m}\frac{h_{n-1,m}}{h_{n,m}}f_{n-1,m}+b_{n,m}\frac{h_{n,m-1}}{h_{n,m}}f_{n,m-1}
\end{align*}
to be symmetric, it must obey the relation
\begin{equation}\label{RenCoeffSym}
\frac{h_{n+1,m}}{2h_{n,m}}=a_{n+1,m}\frac{h_{n,m}}{h_{n+1,m}}, \quad
\frac{h_{n,m+1}}{2h_{n,m}}=b_{n,m+1}\frac{h_{n,m}}{h_{n,m+1}},
\end{equation}
which can be rewritten as follows
\[
h_{n+1,m}^2=2a_{n+1,m}h_{n,m}^2,\quad
h_{n,m+1}^2=2b_{n,m+1}h_{n,m}^2.
\]
Now, we see that for the existence of the family $h_{n,m}\ne 0$ the following compatibility condition must be satisfied
\[
a_{n+1,m}b_{n+1,m+1}=
b_{n,m+1}a_{n+1,m+1}.
\]
The latter relation can be obtained from the consistency relations \eqref{eq:3.3}--\eqref{eq:3.4}. Indeed,
it easily follows from  \eqref{eq:3.3}, \eqref{eq:3.4}, and \eqref{ForSymToEx} that
\[
\frac{a_{n+1,m+1}}{a_{n+1,m}}=
\frac{c_{n+1,m}-d_{n+1,m}}{c_{n,m}-d_{n,m}}=
\frac{c_{n,m+1}-d_{n,m+1}}{c_{n,m}-d_{n,m}}=
\frac{b_{n+1,m+1}}{b_{n,m+1}},
\]
which is exactly what we need. Hence, the table  $h_{n,m}\ne 0$ can be constructed, say,  by the initialization $h_{0,0}=1$.
Finally, noticing that
\[
\sqrt{2{a_{n,m}}}=\frac{h_{n,m}}{h_{n-1,m}}, \quad
\sqrt{2{b_{n,m}}}=\frac{h_{n,m}}{h_{n,m-1}},
\]
we get the operator
\[
\begin{split}
(\Delta_s f)_{n,m}=\sqrt{\frac{a_{n+1,m}}{2}}f_{n+1,m}+\sqrt{\frac{b_{n,m+1}}{2}}f_{n,m+1}+\frac{c_{n,m}+d_{n,m}}{2}f_{n,m}+\\
+\sqrt{\frac{a_{n,m}}{2}}f_{n-1,m}+ \sqrt{\frac{b_{n,m}}{2}}f_{n,m-1},
\end{split}
\]
which can be represented as a sum
\[
\Delta_s=J_1+J_2
\]
of two symmetric Jacobi ``matrices" of the following form
\[
(J_1f)_{n,m}=\sqrt{\frac{a_{n+1,m}}{2}}f_{n+1,m}+\frac{c_{n,m}}{2}f_{n,m}+
\sqrt{\frac{a_{n,m}}{2}}f_{n-1,m},
\]
\[
(J_2f)_{n,m}=\sqrt{\frac{b_{n,m+1}}{2}}f_{n,m+1}+\frac{d_{n,m}}{2}f_{n,m}+
\sqrt{\frac{b_{n,m}}{2}}f_{n,m-1}.
\]
Therefore, $\Delta_s$ is symmetric.
\end{proof}

\noindent \textbf{Remark.}
It should also be noted that the underlying eigenvalue problem reduces to the following one
\[
\Delta_s\pi^{(s)}(z)=z\pi^{(s)}(z),
\]
where we set
\[
\Delta_s=D^{-1}\Delta D, \quad
\pi^{(s)}(z)=D^{-1}\pi(z).
\]
In other words, the multiple orthogonal polynomials $P_{n,m}^{(s)}$ corresponding to the symmetric operator $\Delta_s$
and the multiple orthogonal polynomials $P_{n,m}$ corresponding to the monic operator $\Delta$ are related
in the following manner
\[
P_{n,m}^{(s)}(z)=\frac{1}{h_{n,m}}P_{n,m}(z).
\]
Note that we chose $h_{0,0}=1$ in order to have $P_{0,0}^{(s)}=1$.

It is also easy to give sufficient conditions for a 2D discrete Schr\"odinger operator to have eigenfunctions that consist of multiple orthogonal polynomials.

 \begin{theorem}
Let $\widetilde{\Delta}_s$ have  the form
\begin{align*}
(\widetilde{\Delta}_s f)_{n,m}=&\sqrt{\frac{a_{n+1,m}}{2}}f_{n+1,m}+ \sqrt{\frac{b_{n,m+1}}{2}}f_{n,m+1}+q_{n,m}f_{n,m}\\
&+\ \sqrt{\frac{a_{n,m}}{2}}f_{n-1,m}+ \sqrt{\frac{b_{n,m}}{2}}f_{n,m-1}.
\end{align*}
Suppose that there exist two sets of numbers
$c_{n,m}$ and $d_{n,m}$ such that
\begin{equation}\label{Potential}
q_{n,m}=\frac{c_{n,m}+d_{n,m}}{2},
\end{equation}
and the coefficients $a_{n,m}$, $b_{n,m}$, $c_{n,m}$, $d_{n,m}$ satisfy the consistency conditions \eqref{eq:3.1}--\eqref{eq:3.4} together with \eqref{ForSymToEx}.
Also assume that
\begin{equation}\label{TechCondDetS}
\frac{a_{n+1,m+1}}{a_{n+1,m}}+\frac{b_{n+1,m+1}}{b_{n,m+1}}\ne 2.
\end{equation}
Then the families $a_{n,m}$, $b_{n,m}$, $q_{n,m}$ determine the sets $c_{n,m}$, $d_{n,m}$ uniquely.
In other words, the operator $\widetilde{\Delta}_s$, the consistency conditions \eqref{eq:3.1}--\eqref{eq:3.4}, and  \eqref{ForSymToEx} uniquely define
the operators $H_1$ and $H_2$ of the form \eqref{DifOp1}--\eqref{DifOp2}. Thus, in this case, the operator $\widetilde{\Delta}_s$
generates a family of multiple orthogonal polynomials $\bigl(P_{n,m}\bigr)_{n,m=0}^{\infty}$.
\end{theorem}

\section{Some examples}

In this section we recall a few examples of multiple orthogonal polynomials given in \cite{NA2013}, \cite{vanA2011}
to illustrate our method.

\subsection{Multiple Hermite polynomials}
Multiple Hermite polynomials $H_{n,m}$ are monic polynomials of degree $n+m$ that satisfy
the following orthogonality conditions

\[
\int_{-\infty}^\infty x^k H_{n,m}(x) e^{-x^2+c_1 x}\, dx = 0, \quad k=0,1,\ldots, n-1,
\]
\[
\int_{-\infty}^\infty x^k H_{n,m}(x) e^{-x^2+c_2 x}\, dx = 0, \quad k=0,1,\ldots, m-1,
\]
where $c_1 \ne c_2$. The corresponding recurrence relations are explicitly given as
\[
x H_{n,m}(x) = H_{n+1,m}(x) + \frac{c_1}{2}H_{n,m}(x) +
\frac{n}{2}H_{n-1,m}(x)+\frac{m}{2}H_{n,m-1}(x),
\]
\[
x H_{n,m}(x) = H_{n,m+1}(x) + \frac{c_2}{2}H_{n,m}(x) +
\frac{n}{2}H_{n-1,m}(x)+\frac{m}{2}H_{n,m-1}(x).
\]
So, in this case we have that
\[
a_{n,m}=\frac{n}{2}, \quad
b_{n,m}=\frac{m}{2}, \quad
c_{n,m}=\frac{c_1}{2}, \quad
d_{n,m}=\frac{c_2}{2}.
\]
Hence, the relation \eqref{ForSymToEx} is obviously satisfied and, therefore, we have the symmetric operator
\begin{equation}\label{Hermite2D}
\begin{split}
(\Delta_s f)_{n,m}=&\frac{\sqrt{{n+1}}}{2}f_{n+1,m}+ \frac{\sqrt{{m+1}}}{2}f_{n,m+1}+\frac{c_{1}+c_{2}}{4}f_{n,m}\\
&+\ \frac{\sqrt{{n}}}{2}f_{n-1,m}+\frac{\sqrt{{m}}}{2}f_{n,m-1}.
\end{split}
\end{equation}
From the form of this operator it is clear that one cannot uniquely reconstruct the corresponding multiple orthogonal polynomials. Indeed, we end up with the operator \eqref{Hermite2D}
if we start with any multiple Hermite polynomials $H_{n,m}=H_{n,m}(c_1',c_2')$ such that
$c_1'+c_2'=c_1+c_2$.

\subsection{Multiple Laguerre polynomials of the first kind}
These polynomials are given by the orthogonality relations
\[
\int_0^\infty x^k L_{n,m}(x) x^{\alpha_1} e^{-x}\, dx = 0 , \quad k = 0, 1, \ldots, n-1, \]
\[
\int_0^\infty x^k L_{n,m}(x) x^{\alpha_2} e^{-x}\, dx = 0 , \quad k = 0, 1, \ldots, m-1, \]
where $\alpha_1, \alpha_2 > -1$ and $\alpha_1 - \alpha_2 \notin \dZ$.
For multiple Laguerre polynomials of the first kind it is known that
\[
c_{n,m}=2n+m+\alpha_1+1,\quad d_{n,m}=n+2m+\alpha_2+1,
\]
which do not satisfy \eqref{ForSymToEx}. Thus, the underlying cross-shaped operator cannot be symmetrized.

\subsection{Multiple Meixner polynomials of the first kind}
The multiple Meixner polynomials of the first kind $M_{n,m}^{(1)}$
are the monic polynomials of degree $n+m$ for which
\[
\sum_{k=0}^\infty M_{n,m}^{(1)}(k) k^\ell \frac{(\beta)_k(c_1)^k}{k!} = 0, \quad \ell=0,1,\ldots,n-1,
\]
\[
\sum_{k=0}^\infty M_{n,m}^{(1)}(k) k^\ell \frac{(\beta)_k(c_2)^k}{k!} = 0, \quad \ell=0,1,\ldots,m-1,
\]
where $\beta > 0$ and $0 < c_1 \ne c_2 < 1$. In this case, the nearest neighbour recurrence relation takes the form
\begin{multline*}
   x M_{n,m}^{(1)}(x) = M_{n+1,m}^{(1)}(x) +\Bigg( (\beta +n+m)\frac{c_1}{1-c_1}+
   \frac{n}{1-c_1}+\frac{m}{1-c_2} \Bigg )M_{n,m}^{(1)}(x) \\
  + \frac{c_1 n}{(1-c_1)^2}(\beta+n+m-1) M_{n-1,m}^{(1)}(x)
  + \frac{c_2 m}{(1-c_2)^2}(\beta+n+m-1) M_{n,m-1}^{(1)}(x),
\end{multline*}
\begin{multline*}
   x M_{n,m}^{(1)}(x) = M_{n,m+1}^{(1)}(x) +\Bigg( (\beta +n+m)\frac{c_2}{1-c_2}+
   \frac{n}{1-c_1}+\frac{m}{1-c_2} \Bigg )M_{n,m}^{(1)}(x) \\
  + \frac{c_1 n}{(1-c_1)^2}(\beta+n+m-1) M_{n-1,m}^{(1)}(x)
  + \frac{c_2 m}{(1-c_2)^2}(\beta+n+m-1) M_{n,m-1}^{(1)}(x).
\end{multline*}
Since the relation \eqref{ForSymToEx} is true in this case, we obtain the following
2D Schr\"odinger operator
\begin{equation}\label{Meixner2D}
\begin{split}
(\Delta_s f)_{n,m}=&\frac{\sqrt{c_1(n+1)(n+m+\beta)}}{\sqrt{2}(1-c_1)}f_{n+1,m}+
\frac{\sqrt{c_2(m+1)(n+m+\beta)}}{\sqrt{2}(1-c_2)}f_{n,m+1}\\
&+ \left(\frac{(\beta +n+m)}{2}\left(\frac{c_1}{1-c_1}+\frac{c_2}{1-c_2}\right)+\frac{n}{1-c_1}+\frac{m}{1-c_2}\right)f_{n,m}\\
&+\ \frac{\sqrt{c_1n(n+m+\beta-1)}}{\sqrt{2}(1-c_1)}f_{n-1,m}+
\frac{\sqrt{c_2m(n+m+\beta-1)}}{\sqrt{2}(1-c_2)}f_{n,m-1}.
\end{split}
\end{equation}
Also, the recurrence coefficients satisfy \eqref{TechCondDetS}, that is
\begin{equation*}
\frac{a_{n+1,m+1}}{a_{n+1,m}}+\frac{b_{n+1,m+1}}{b_{n,m+1}}=
2\frac{\beta+n+m+1}{\beta+n+m}=2\left(1+\frac{1}{\beta+n+m}\right)>2.
\end{equation*}
This means that the 2D Schr\"odinger operator \eqref{Meixner2D} determines the corresponding
multiple orthogonal polynomials uniquely.

\end{document}